\input louC.sty

\documentclass[a4paper,draft]{amsproc}
\usepackage{amsrefs}
\usepackage{amsmath}
\usepackage{mathrsfs}
\usepackage{amssymb}
\usepackage{amscd} 

\theoremstyle{plain}
 \newtheorem{thm}{\textbf{Theorem}}[section]
 \newtheorem{prop}{\textbf{Proposition}}[section]
 \newtheorem{lem}{\textbf{Lemma}}[section]
 \newtheorem{cor}{\textbf{Corollary}}[section]
\theoremstyle{definition}

\theoremstyle{remark}
 \newtheorem{rem}{\textbf{Remark}}[section]
 \numberwithin{equation}{section}

\title[Compact Product of Hankel and Toeplitz Operators]{Compact Product of Hankel and Toeplitz\\ Operators on the Hardy space}

\keywords{Hankel operator, Toeplitz operator, Hardy space}

\author[Chu]{ Cheng Chu}

\address{
Department of Mathematics \\ 
Washington University in Saint Louis  \\ 
Saint Louis, Missouri \\
USA}
\email{chengchu@math.wustl.edu}

\subjclass {Primary 47; Secondary 30}

\begin{document}
\bibliographystyle{plain}

\vspace{18mm}
\setcounter{page}{1}
\thispagestyle{empty}

\begin{abstract}
In this paper, we study the product of a Hankel operator and a Toeplitz operator on the Hardy space. We give necessary and sufficient conditions of when such a product $H_f T_g$ is compact.
\end{abstract}

\maketitle

\section{Introduction}  

Let $\DD$ be the open unit disk in the complex plane. Let $L^2$ denote the Lebesgue space of square integrable functions on the unit circle $\partial\DD$. The Hardy space $H^2$ is the subspace of $L^2$ of analytic functions on $\DD$. Let $P$ be the orthogonal projection from $L^2$ to $H^2$. For $f\in L^\infty$, the Toeplitz operator $T_f$ and the Hankel operator $H_f$ with symbol $f$ are defined by $$T_fh=P(fh),$$ and $$H_fh=PU(fh),$$ for $h\in H^2$. Here $U$ is the unitary operator on $L^2$ defined by $$Uh(z)=\bz \tih (z),$$
where $\tif(z)=f(\bz)$.
Clearly, $$H^*_f=H_{f^*},$$ where $f^*(z)= \overline{f(\bz)}.$

Hankel operator can also be defined by $$\cH _fh=(I-P)(fh),$$
and it is easy to verify that $H_f=U\cH_f$.

Let us first look at the compactness of Toeplitz and Hankel operators individually. The only compact Toeplitz operator is the zero operator (see for example \cite{dou72}, \cite{zhu}). For the Hankel operator, we have the following theorem (see for example \cite{pel02}, \cite{zhu}), usually referred to as Hartman's Criterion.
\begin{thm}
Let $f\in L^\infty$. Then the Hankel operator $H_f$ is compact if and only if $$f\in H^\infty+C.$$
Here $C$ denotes the space of continuous functions on the unit circle. $H^\infty+C$ is the linear span of $H^\infty$ and $C$. It is a closed subalgebra of $L^\infty$ containing $H^\infty$ \mbox{(see \cite{sar67})}.
\end{thm}

The problem of characterizing the compactness for the product of two Hankel operators turns out to be much more difficult. Axler, Chang, Sarason \cite{ax78}, and Volberg \cite{vol82} gave necessary and sufficient conditions that the product of two Hankel operators is compact. They proved the following result:
\begin{thm}\label{ACS}
Let $f,g\in L^\infty$. $H_{\tif}H_g$ is compact if and only if
\beq\label{AC}H^\infty[\barf] \cap H^\infty[g]\subset H^\infty+C.\eeq
Here $H^\infty[f]$ denotes the closed subalgebra of $L^\infty$ generated by $H^\infty$ and $f$.
\end{thm}
They also gave a local version of the algebraic condition \eqref{AC} using the notion of support sets. We will define the support sets in Section 2.
\begin{thm}\label{LC}
Let $f,g\in L^\infty$. $$H^\infty[\barf] \cap H^\infty[g]\subset H^\infty+C$$ if and only if for each support set $S$, either $\barf|_S$ or $g|_S$ is in $H^\infty|_S$.
\end{thm}

Later, Zheng in \cite{zh96} gave the following elementary condition that also characterizes the compactness of $H_{\tif}H_g$.
\begin{thm}\label{EC}
Let $f,g\in L^\infty$. $H_{\tif}H_g$ is compact if and only if $$\lim_{|z|\to 1^-}||H_{\barf}k_z||\cdot||H_gk_z||=0.$$
Here $k_z$ denotes the normalized reproducing kernel at $z$.
\end{thm}
The relations between these three conditions in Theorem \ref{ACS}, \ref{LC} and \ref{EC} can be found in Section 3 and 4. Inspired by the above theorems, we consider the product of a Hankel operator and a Toeplitz operator in this paper. The following theorem is our main result:
\begin{thm}\label{M}
Let $f,g\in L^\infty$. The product $K=H_fT_g$ of the Hankel operator $H_f$ and the Toeplitz operator $T_g$ is compact if and only if for each support set $S$, one of the following holds:
\begin{enumerate}
\item $f|_S\in H^\infty|_S.$\\
\item $g|_S\in H^\infty|_S$ and $(fg)|_S\in H^\infty|_S.$
\end{enumerate}
\end{thm}
Analogously to Theorem \ref{ACS}, we also obtain the following algebraic version of Theorem \ref{M}:
\begin{thm}\label{M2}
$$H^\infty[f] \cap H^\infty[g, fg]\subset H^\infty+C$$ if and only if for each support set $S$, one of the following holds:
\begin{enumerate}
\item $f|_S\in H^\infty|_S.$\\
\item $g|_S\in H^\infty|_S$ and $(fg)|_S\in H^\infty|_S.$
\end{enumerate}
\end{thm}

In Section \ref{sec}, we will give a generalization of Theorem \ref{M} to the sum of two products of Hankel and Toeplitz operators.

\section{Preliminaries}
We begin this section by establishing the relation between Toeplitz operators and Hankel operators. Consider the multiplication operator $M_f$ on $L^2$ for $f\in L^\infty$, defined by $M_fh=fh$. $M_f$ can be expressed as an operator matrix with respect to the decomposition $L^2=H^2\oplus (H^2)^\perp$ as the following:
$$
M_f=
\begin{pmatrix}
T_f & H_{\tilde{f}}U\\
UH_f & UT_{\tilde{f}}U
\end{pmatrix}
$$
For $f,g\in L^\infty$, $M_{fg}=M_fM_g$, so multiplying the matrices and comparing the entries, we get:
\begin{prop}\label{prop}
Let $f$ and $g$ be in $L^\infty$. Then
\begin{enumerate}
\item $T_{fg}=T_fT_g+H_{\tif} H_g$.
\item $H_{fg}=H_fT_g+T_{\tif} H_g$.
\item If $g\in H^\infty$, then $H_fT_g=T_{\tif}H_g$.
\end{enumerate}
\end{prop}

Let $x,y\in L^2$. Define $x\otimes y$ to be the following rank one operator on $L^2$:
$$
(x\otimes y)(f)=\langle f,y\rangle x.
$$
\begin{prop}
Let $x,y\in L^2$ and let $S,T$ be operators on $L^2$. Then
\begin{enumerate}
\item $(x\otimes y)^*=y\otimes x$.
\item $||x\otimes y||=||x||\cdot||y||$.
\item $S(x\otimes y)T=(Sx)\otimes(T^* y)$.
\end{enumerate}
\end{prop}

For each $z\in \DD$, let $k_z$ denote the normalized reproducing kernel at z: $$k_z(w)=\frac{\sqrt{1-|z|^2}}{1-\bz w},$$
and $\phi_z$ be the M\"obius transform: $$\phi_z(w)=\frac{z-w}{1-\bz w}.$$
We have the following identities:
\begin{lem}\label{ID}
For $z\in \DD$,
\begin{enumerate}
\item $T_{\phi_z}T_{\bar{\phi_z}}=1-k_{z}\otimes k_{z}.$\\
\item $T^*_{\tilde{\phi_z}}T_{\tilde{\phi_z}}=1-k_{\bz}\otimes k_{\bz}.$\\
\item $H_{\bar{\phi_z}}=-k_{\bz}\otimes k_z.$
\end{enumerate}
\end{lem}
These identities can be found in ~\cite{guo03},~\cite{xia00} and~\cite{zh96}.

The next lemma in \cite{cgiz05} and \cite{guo03} gives a relation between $H_f$ and $H^*_f$.
\begin{lem}\label{*}
Let $f\in L^\infty$, $g \in H^2$. Then $H^*_f g^*=(H_f g)^*$ and thus $||H^*_f g^*||=||H_f g||$. In particular, $||H^*_f k_{\bz}||=||H_f k_z||$.
\end{lem}
\begin{proof}
Notice that for all $g\in L^2$, $(Ug)^*=Ug^*$ and $Pg^*=(Pg)^*$. Thus
$$H^*_f g^*=H_{f^*}g^*=PU(f^*g^*)=P(Ufg)^*=(PUfg)^*=(H_fg)^*$$
Since $||h||=||h^*||$ for all $h\in L^2$, we get $||H^*_f g^*||=||H_f g||$.
\end{proof}

To state the local conditions, we need some notation for the maximal ideal space. For a uniform algebra $B$, let $M(B)$ denotes the maximal ideal space of $B$, the space of nonzero multiplicative linear functionals of $B$. Given the weak-star topology of $B^*$, which is called the Gelfand topology, $M(B)$ is a compact Hausdorff space.

We identify $\DD$ in the usual way as a subset of $M(H^\infty)$ (see for example \cite{gar81}). By Carleson's Corona Theorem \cite{car62}, $\DD$ is dense in $M(H^\infty)$. Moreover, $M(H^\infty+C)=M(H^\infty)\backslash \DD$ (see\cite{sar67}).

For any $m$ in $M(H^\infty)$, there exists a representing measure $\mu_m$ such that $m(f)=\int fd\mu_m$, for all $f\in H^\infty$ (see \cite{gar81}*{Chapter V}). A subset $S$ of $M(H^\infty)$ is called a support set if it is the support of a representing measure for a functional in $M(H^\infty+C)$.

\section{Proof of Theorem \ref{M2}}
In this section we prove Theorem \ref{M2}. The proof we present here is analogous to the proof of \cite{gor99}*{Lemma 1.1}.

The proof is based on the following two lemmas:
\begin{lem}\cite{gor99}*{Lemma 1.3}\label{1.3}
Let ${A_\Ga}$ be a family of Douglas algebras. Then $$M(\cap A_{\Ga})=\overline{\cup M(A_{\Ga})}.$$
\end{lem}
\begin{lem}\cite{gor99}*{Lemma 1.5}\label{1.4}
Let $m\in M(H^\infty+C)$ and let $S$ be the support set of m. Then $m\in M(H^\infty[f])$ if and only if $f|_S\in H^\infty|_S$.
\end{lem}

{\bf Proof of Theorem \ref{M2}.} Let $$A=H^\infty[f] \cap H^\infty[g, fg].$$ By Lemma \ref{1.3},
$$
M(A)=M(H^\infty[f]) \cup M(H^\infty[g, fg]).
$$
Suppose $$H^\infty[f] \cap H^\infty[g, fg]\subset H^\infty+C.$$ Then $A\subset H^\infty+C$, and so $M(H^\infty+C)\subset M(A).$
Let $m\in M(H^\infty+C)$. Lemma \ref{1.4} gives that either condition (1) or condition (2) holds.

Conversely, let $S$ be the support set for $m\in M(H^\infty+C)$ and suppose one of the Conditions (1) and (2) holds for $m$. Then by Lemma \ref{1.4}, either
$m\in M(H^\infty[f])$ or $m\in M(H^\infty[g, fg])$. Thus, $M(H^\infty+C)\subset M(A).$ By Chang-Marshall Theorem \cite{cha76,mar76}, $A=H^\infty+C$. $\Box$

\section{Compact Operators and Local Condition}
In this section, we present the main tools in the proof of Theorem \ref{M}.

The following lemma in \cite{guo03}*{Lemma 9} gives a nice property of compact operators.
\begin{lem}\label{ML}
If $K: H^2\to H^2$ is a compact operator, then
\beq\label{ml}
\lim_{|z|\to 1^-}||K-T_{\tilde{\phi_z}}KT_{\bar\phi_z}||=0.
\eeq
\end{lem}
\begin{rem}
By the Corona Theorem, \eqref{ml} can be restated as the following:\\
For each $m\in M(H^\infty+C)$, there is a net $z\to m$ such that
$$
\lim_{z\to m}||K-T_{\tilde{\phi_z}}KT_{\bar\phi_z}||=0.
$$
\end{rem}

In \cite{guo05}, Guo and Zheng used the distribution function inequality to prove the following theorem, which can be viewed as a partial converse of Lemma \ref{ML}.
\begin{thm}\label{gz}
Let $T$ be a finite sum of finite products of Toeplitz operators. Then $T$ is a compact perturbation of a Toeplitz operator if and only if
$$
\lim_{|z|\to 1^-}||T-T^*_{\phi_z}TT_{\phi_z}||=0.
$$
\end{thm}
\begin{rem}
Theorem \ref{gz} cannot be applied directly to $H_fT_g$, since $H_fT_g$ might not be a finite sum of finite products of Toeplitz operators. However, by Proposition \ref{prop},
$$
(H_fT_g)^*(H_fT_g)=T_{\bg}H^*_fH_fT_g=T_{\bg}(T_{{\barf}f}-T_{\barf}T_f)T_g,
$$
thus $(H_fT_g)^*(H_fT_g)$ is a finite sum of finite products of Toeplitz operators.
\end{rem}
\begin{rem}
The symbol map $\Gs$ that sends every Toepltiz operator $T_\phi$ to its symbol $\phi$ was introduced in \cite{dou72} and can be defined on the Toeplitz algebra, the closed algebra generated by Toeplitz operators. Barr\'ia and Halmos in \cite{ba82} showed that $\Gs$ can be extended to a $*$-homomorphism on the Hankel algebra, the closed algebra generated by Toeplitz and Hankel operators. And they also showed that the symbols of compact operators and Hankel operators are zero. Note that $(H_fT_g)^*(H_fT_g)$ has symbol zero, so it is a compact perturbation of a Toeplitz operator if and only if it is compact.
\end{rem}
By Theorem \ref{gz} and above remarks, we have
\begin{cor}\label{cor}
$K=H_fT_g$ is compact if and only if
$$
\lim_{|z|\to 1^-}||K^*K-T^*_{\phi_z}K^*KT_{\phi_z}||=0.
$$
\end{cor}

The following lemma from \cite{gor99}*{Lemma 2.5, 2.6} which interprets the local condition in an elementary way, will be used several times later.
\begin{lem}\label{lim}
Let $f\in L^\infty$, $m\in M(H^\infty+C)$, and let $S$ be the support set of $m$. Then the following are equivalent:
\begin{enumerate}
\item $f|_S\in H^\infty|_S.$
\item $\varliminf\limits_{z\to m}||H_f k_z||=0.$
\item $\lim\limits_{z\to m}||H_f k_z||=0.$
\end{enumerate}
\end{lem}

We also need the following technical lemma.
\begin{lem}\label{*lim} \cite{guo03}*{Lemma 17,18}
Let $f,g\in L^\infty$, $m\in M(H^\infty+C)$.
\begin{enumerate}
\item If
$$\lim_{z\to m}||H_f k_z||=0,$$ then $$\lim_{z\to m}||H_fT_g k_z||=0.$$
\item If
$$\lim_{z\to m}||H^*_f k_{\bz}||=0,$$ then $$\lim_{z\to m}||H^*_fT_g k_{\bz}||=0.$$
\end{enumerate}
\end{lem}

\section{Proof of the Main Theorem}
In this section, we prove Theorem \ref{M}. First we set up the following two identities:
\begin{lem}\label{ML2}\cite{guo03}*{Lemma 6}
Let $f,g\in L^\infty$ and $z\in\DD$. Then
$$
T_{\tilde{\phi_z}}H_fT_gT_{\bar{\phi_z}}=H_fT_g-(H_fT_gk_z)\otimes k_z+(H_fk_z)\otimes(T_{\phi_z}H^*_fk_{\bz})
$$
\end{lem}

\begin{lem}\label{ML3}
Let $f,g\in L^\infty$, $z\in\DD$ and $K=H_fT_g$. Then
$$
KT_{\phi_z}=T_{\tilde{\phi_z}}K-(H_fk_z)\otimes(H^*_gk_{\bz}).
$$
\end{lem}
\begin{proof}
Since $\phi_z\in H^\infty$, by Proposition \ref{prop},
$$
T_gT_{\phi z}=T_{\phi_z}T_g+H_{\tilde{\phi_z}}H_g,
$$
and $$T_{\tilde{\phi_z}}H_g=H_gT_{\phi_z}.$$
Thus, \begin{align*}
KT_{\phi_z}&=H_fT_gT_{\phi_z}=H_fT_{\phi_z}T_g+H_fH_{\tilde{\phi_z}}H_g\\
&=T_{\tilde{\phi_z}}H_fT_g+H_fH_{\tilde{\phi_z}}H_g\\
&=T_{\tilde{\phi_z}}K-(H_fk_z)\otimes(H^*_gk_{\bz}).
\end{align*}
The last equality follows from Lemma \ref{ID}.
\end{proof}

Now we are ready to prove Theorem \ref{M}.\\
{\bf Proof of Theorem \ref{M}.} Necessity: Suppose K is compact. By Lemma \ref{ML}, we have:
$$
\lim_{|z|\to 1^-}||H_fT_g-T_{\tilde{\phi_z}}H_fT_gT_{\bar\phi_z}||=0.
$$
By Lemma \ref{ML2}, we have
$$
||H_fT_g-T_{\tilde{\phi_z}}H_fT_gT_{\bar\phi_z}||=||(H_fT_gk_z)\otimes k_z-(H_fk_z)\otimes(T_{\phi_z}H^*_gk_{\bz})||
$$
Since $k_z\to 0$ weakly as $|z|\to 1$ and $H_fT_g$ is compact,
\beq\label{pf1}
||H_fT_gk_z||\to 0.
\eeq
So $$\lim_{|z|\to 1^-}||(H_fk_z)\otimes(T_{\phi_z}H^*_gk_{\bz})||=0.$$
Since
\begin{align*}
&||(H_fk_z)\otimes(H^*_gk_{\bz})||=||((H_fk_z)\otimes(T_{\phi_z}H^*_gk_{\bz}))T_{\phi_z}||\\
\leq&||(H_fk_z)\otimes(T_{\phi_z}H^*_gk_{\bz})||\cdot||T_{\phi_z}||\leq||(H_fk_z)\otimes(T_{\phi_z}H^*_gk_{\bz})||,
\end{align*}
we get $$\lim_{|z|\to 1^-}||(H_fk_z)\otimes(H^*_gk_{\bz})||=0.$$
By Lemma \ref{*}, $$\lim_{|z|\to 1^-}||H_fk_z||\cdot||H_gk_z||=0.$$
Let $m\in M(H^\infty+C)$ and let $S$ be the support set of $m$. By the Corona Theorem, there is a net $z$ converging to m, and
$$\lim_{z\to m}||H_fk_z||\cdot||H_gk_z||=0.$$
Thus, either $$\varliminf_{z\to m}||H_fk_z||=0$$ or \beq\label{pf2} \varliminf_{z\to m}||H_gk_z||=0.\eeq
By Lemma \ref{lim}, we have $f|_S\in H^\infty|_S$ or $g|_S\in H^\infty|_S$. In the second case,we have
\begin{align*}
&\varliminf_{z\to m}||H_{fg}k_z||=\varliminf_{z\to m}||H_fT_gk_z+T_{\tif} H_gk_z||\\
\leq&\varliminf_{z\to m}||H_fT_gk_z||+||T_{\tif}||\cdot\varliminf_{z\to m}||H_gk_z||=0.
\end{align*}
The first inequality comes from Proposition \ref{prop} and the last equality follows from \eqref{pf1} and \eqref{pf2}.

Therefore, Lemma \ref{lim} implies $(fg)|_S\in H^\infty|_S$.

Sufficiency: By Corollary \ref{cor}, we need to show: for any $m\in M(H^\infty+C)$, \beq\label{pf3}\lim_{z\to m}||K^*K-T^*_{\phi_z}K^*KT_{\phi_z}||=0.\eeq
Let $F_z=-(H_fk_z)\otimes(H^*_gk_{\bz})$. Lemma \ref{ML3} gives $$KT_{\phi_z}=T_{\tilde{\phi_z}}K+F_z.$$
Then
\begin{align}\label{pf4}
\nnb T^*_{\phi_z}K^*KT_{\phi_z}&=(KT_{\phi_z})^*(KT_{\phi_z})=K^*T^*_{\tilde{\phi_z}}T_{\tilde{\phi_z}}K+(K^*T^*_{\tilde{\phi_z}})F_z+F^*_z(T_{\tilde{\phi_z}}K)+F^*_zF_z\\
&=K^*K+(K^*k_{\bz})\otimes(K^*k_{\bz})+(K^*T^*_{\tilde{\phi_z}})F_z+F^*_z(T_{\tilde{\phi_z}}K)+F^*_zF_z.
\end{align}
The last equality comes from Lemma \ref{ID} (2).

Let $S$ be the support set of $m$. If Condition (1) holds, i.e., $f|_S\in H^\infty|_S$, Lemma \ref{lim} and Lemma \ref{*} give
$$\lim_{z\to m}||H_fk_z||=0,$$ and $$\lim_{z\to m}||H^*_fk_{\bz}||=0.$$
So \beq\label{pf5}\lim_{z\to m}||F_z||=0,\eeq and $$\lim_{z\to m}||K^*k_{\bz}||=\lim_{z\to m}||T^*_gH^*_fk_{\bz}||=0.$$
Since $||K||<\infty$ and $\sup\limits_{z\in\DD}||F_z||<\infty$, \eqref{pf4} implies \eqref{pf3}.\\
If Condition (2) holds, i.e., $g|_S\in H^\infty|_S$ and $(fg)|_S\in H^\infty|_S$, by Lemma \ref{lim}, $$\lim_{z\to m}||H_gk_z||=0,$$ and
\beq\label{pf6}\lim_{z\to m}||H_{fg}k_z||=0.\eeq
So \eqref{pf5} also holds.
By Proposition \ref{prop}, $$(H_fT_g)^*k_{\bz}=H^*_{fg}k_{\bz}-(T_{\tif} H_g)^*k_{\bz}=H^*_{fg}k_{\bz}-H^*_g T_{f^*}k_{\bz}.$$
Using \eqref{pf5} and Lemma \ref{*lim}, we get $$\lim_{z\to m}||K^*k_{\bz}||=0.$$
Thus, \eqref{pf3} holds and $H_fT_g$ is compact. $\Box$
\\

Notice that $(T_fH_g)^*=H_{g^*}T_{\barf}$. Combining Theorem \ref{M} and Theorem \ref{M2}, we get the following characterization of the compactness of the product $T_fH_g$:
\begin{cor}
Let $f,g\in L^\infty$. The following are equivalent:
\begin{enumerate}
\item $T_fH_g$ is compact.\\
\item $H^\infty[g^*] \cap H^\infty[\barf, \barf g^*]\subset H^\infty+C.$\\
\item For each support set $S$, one of the following holds:
\begin{enumerate}
\item $g^*|_S\in H^\infty|_S.$\\
\item $\barf|_S\in H^\infty|_S$ and $(\barf g^*)|_S\in H^\infty|_S.$
\end{enumerate}
\end{enumerate}
\end{cor}

\section{A Generalization}\label{sec}
In this section, we prove a generalization of Theorem \ref{M} that characterizes the compactness of the sum of two products of Hankel and Toeplitz operators.
\begin{thm}
Suppose $f_1,f_2,g_1,g_2$ are in $L^\infty$. Then $K=H_{f_1}T_{g_1}+H_{f_2}T_{g_2}$ is compact if and only if for each support set $S$, one of the following holds:
\begin{enumerate}
\item $f_1|_S$, and $f_2|_S$ are in $H^\infty|_S$.
\item $f_1|_S$, $g_2|_S$, and $(f_2g_2)|_S$ are in $H^\infty|_S$.
\item $g_1|_S$, $(f_1g_1)|_S$ and $f_2|_S$ are in $H^\infty|_S$.
\item $g_1|_S$, $g_2|_S$ $(f_1g_1)|_S$ and $(f_2g_2)|_S$ are in $H^\infty|_S$.
\item There exists nonzero constant $c$ such that $(cf_1+f_2)|_S, (g_1-cg_2)|_S$, and $[f_1(g_1-cg_2)]|_S$ are all in $H^\infty|_S$.
\end{enumerate}
\end{thm}
\begin{proof}
Necessity: Suppose $K$ is compact. By Lemma \ref{ML} and Lemma \ref{ML2}, as in the proof of Theorem \ref{M}, we have
$$
\lim_{|z|\to 1^-}||(H_{f_1}k_z)\otimes(H^*_{g_1}k_{\bz})+(H_{f_2}k_z)\otimes(H^*_{g_2}k_{\bz})||=0.
$$

Let $m\in M(H^\infty+C)$ and let $S$ be the support set of m. By the Corona Theorem, there exist a net $z\to m$ such that
\beq\label{G1}
\lim_{z\to m}||(H_{f_1}k_z)\otimes(H^*_{g_1}k_{\bz})+(H_{f_2}k_z)\otimes(H^*_{g_2}k_{\bz})||=0.
\eeq

Now we consider three cases.

CASE 1: $$\varliminf_{z\to m}||H_{f_1}k_z||=0~\m{or}\varliminf_{z\to m}||H_{f_2}k_z||=0.$$
By symmetry, we only need to consider $$\varliminf_{z\to m}||H_{f_1}k_z||=0.$$
Lemma \ref{lim} gives $f_1|_S\in H^\infty|_S$, and \eqref{G1} implies $$\varliminf_{z\to m}||(H_{f_2}k_z)\otimes(H^*_{g_2}k_{\bz})||=0.$$
By the proof of Theorem \ref{M},  Condition (1) or (2) holds.

CASE 2: $$\varliminf_{z\to m}||H_{g_1}k_z||=0~\mbox{or}\varliminf_{z\to m}||H_{g_2}k_z||=0.$$
Again, by symmetry, we only consider the case \beq\label{g2}\varliminf_{z\to m}||H_{g_1}k_z||=0.\eeq
Lemma \ref{lim} gives $g_1|_S\in H^\infty|_S$, and \eqref{G1} implies $$\varliminf_{z\to m}||(H_{f_2}k_z)\otimes(H^*_{g_2}k_{\bz})||=0.$$
From the proof of Theorem \ref{M},
\beq\label{C1} f_2|_S\in H^\infty|_S\eeq or \beq\label{C2} g_2|_S\in H^\infty|_S~\m{and}~(f_2g_2)|_S\in H^\infty|_S.\eeq
If \eqref{C1} holds, then by the discussion in CASE 1, Condition (1) or (3) holds.\\
If \eqref{C2} holds, by Lemma \ref{lim}, we have
\beq\label{d1}
\lim_{z\to m}||H_{g_2}k_z||=0
\eeq and
\beq\label{d2}
\lim_{z\to m}||H_{f_2g_2}k_z||=0.
\eeq
Using the identity $$H_{f_1}T_{g_1}+H_{f_2}T_{g_2}=H_{f_1g_1}-T_{\tilde{f_1}}H_{g_1}+H_{f_2g_2}-T_{\tilde{f_2}}H_{g_2}$$
we get $$\varliminf_{z\to m}||H_{f_1g_1}k_z||=\varliminf_{z\to m}||Kk_z+T_{\tilde{f_1}}H_{g_1}k_z-H_{f_2g_2}k_z+T_{\tilde{f_2}}H_{g_2}k_z||=0.$$
The last equality comes from \eqref{g2}, \eqref{d1}, \eqref{d2} and the compactness of $K$. Thus by Lemma \ref{lim} $$(f_1g_1)|_S\in H^\infty|_S$$ and thus Condition (4) holds.

CASE 3: None of the above cases happen, so $$\varliminf_{z\to m}||H_{f_1}k_z||\geq \Gd>0.$$
Applying the operator $[(H_{f_1}k_z)\otimes(H^*_{g_1}k_{\bz})+(H_{f_2}k_z)\otimes(H^*_{g_2}k_{\bz})]^*$ to $H_{f_1}k_z$ we get,
$$
\lim_{z\to m}||~||H_{f_1}k_z||^2\cdot H^*_{g_1}k_{\bz}+\langle H_{f_1}k_z, H_{f_2}k_z\rangle \cdot H^*_{g_2}k_{\bz}~||=0.
$$
Let $$t_z={\langle H_{f_1}k_z, H_{f_2}k_z\rangle}/{||H_{f_1}k_z||^2}.$$ Then $|t_z|\leq {1\over \Gd}$, and we may assume $t_z\to t$ for some constant $t$.
Thus
\begin{align*}
0&=\lim_{z\to m}||H^*_{g_1}k_{\bz}+t_z H^*_{g_2}k_{\bz}||=\lim_{z\to m}||H^*_{g_1}k_{\bz}+ t H^*_{g_2}k_{\bz}||\\
&=\lim_{z\to m}||H^*_{g_1+{\bt}g_2}k_{\bz}||=\lim_{z\to m}||H_{g_1+{\bt}g_2}k_{z}||.
\end{align*}
If $t=0$, then we go back to CASE 2. Now assume $t\neq 0$, let $c=-\bt$, then
\beq\label{C3}
\lim_{z\to m}||H_{g_1-cg_2}k_{z}||=0.
\eeq
By Lemma \ref{lim}, we obtain $$(g_1-cg_2)|_S\in H^\infty|_S.$$
Notice that
\begin{align}\label{id1}
 &(H_{f_1}k_z)\otimes(H^*_{g_1}k_{\bz})+(H_{f_2}k_z)\otimes(H^*_{g_2}k_{\bz})\\
\nnb=&(H_{f_1}k_z)\otimes(H^*_{g_1-cg_2}k_{\bz})+(H_{cf_1+f_2}k_z)\otimes(H^*_{g_2}k_{\bz})
\end{align}
Because \eqref{G1} and \eqref{C3}, we get
$$
\lim_{z\to m}||(H_{cf_1+f_2}k_z)\otimes(H^*_{g_2}k_{\bz})||=0.
$$
Since we assumed in this case $$\varliminf_{z\to m}||H^*_{g_2}k_{\bz}||=\varliminf_{z\to m}||H_{g_2}k_z||>0,$$
we have $$\varliminf_{z\to m}||H_{cf_1+f_2}k_z||=0.$$
Thus Lemma \ref{lim} and Lemma \ref{*lim} imply that $$(cf_1+f_2)|_S\in H^\infty|_S$$ and
\beq\label{C4}\lim_{z\to m}||H_{cf_1+f_2}T_{g_2}k_z||=0.\eeq
Then we use the identity
\begin{align}\label{id2}
K=&H_{f_1}T_{g_1}+H_{f_2}T_{g_2}=H_{f_1}T_{g_1-cg_2}+H_{cf_1+f_2}T_{g_2}\\
\nnb=&H_{f_1(g_1-cg_2)}-T_{\tilde{f_1}}H_{g_1-cg_2}+H_{cf_1+f_2}T_{g_2}.
\end{align}
By \eqref{C3}, \eqref{C4} and the compactness of $K$, we have $$\lim_{z\to m}||H_{f_1(g_1-cg_2)}k_z||=0.$$
Therefore, Lemma \ref{lim} implies $(f_1(g_1-cg_2))|_S\in H^\infty|_S$, which gives Condition (5).

Sufficiency: To prove the converse, we will use Corollary \ref{cor} as in the proof of Theorem \ref{M}. It suffices to show: for any $m\in M(H^\infty+C)$,  \beq\label{G2}\lim_{z\to m}||K^*K-T^*_{\phi_z}K^*KT_{\phi_z}||=0.\eeq
By Lemma \ref{ML3},
$$
KT_{\phi_z}=T_{\tilde{\phi_z}}K-(H_{f_1}k_z)\otimes(H^*_{g_1}k_{\bz})-(H_{f_2}k_z)\otimes(H^*_{g_2}k_{\bz}).
$$
Denote
$$
F_z=-(H_{f_1}k_z)\otimes(H^*_{g_1}k_{\bz})-(H_{f_2}k_z)\otimes(H^*_{g_2}k_{\bz}).
$$
Similarly to the proof of Theorem \ref{M}, we also have $$KT_{\phi_z}=T_{\tilde{\phi_z}}K+F_z,$$ thus
\beq T^*_{\phi_z}K^*KT_{\phi_z}=K^*K+(K^*k_{\bz})\otimes(K^*k_{\bz})+(K^*T^*_{\tilde{\phi_z}})F_z+F^*_z(T_{\tilde{\phi_z}}K)+F^*_zF_z.
\eeq
We need to show
\beq\label{S1}\lim_{z\to m}||F_z||=0\eeq and \beq\label{S2}\lim_{z\to m}||K^*k_{\bz}||=0.\eeq
From the proof of Theorem \ref{M}, if one of the the Conditions (1),(2),(3),(4) holds, we can get \eqref{S1} and \eqref{S2}.

Now suppose Condition (5) holds. By Lemma \ref{lim}, we have
\beq\label{51}
\lim_{z\to m}||H_{g_1-cg_2}k_{z}||=0,
\eeq
\beq\label{52}
\lim_{z\to m}||H_{cf_1+f_2}k_z||=0,
\eeq
\beq\label{53}
\lim_{z\to m}||H_{f_1(g_1-cg_2)}k_z||=0.
\eeq
By \eqref{id1} and Lemma \ref{*}, we have
\begin{align*}
||F_z||&=||(H_{f_1}k_z)\otimes(H^*_{g_1-cg_2}k_{\bz})+(H_{cf_1+f_2}k_z)\otimes(H^*_{g_2}k_{\bz})||\\
&\leq ||H_{f_1}k_z||\cdot||H^*_{g_1-cg_2}k_{\bz}||+||H_{cf_1+f_2}k_z||\cdot||H^*_{g_2}k_{\bz}||\to0.
\end{align*}
as $z\to m$.
Also \eqref{id2} implies
\begin{align*}
||K^*k_{\bz}||&=||H^*_{f_1(g_1-cg_2)}k_{\bz}-H^*_{g_1-cg_2}T^*_{\tilde{f_1}}k_{\bz}+T^*_{g_2}H^*_{cf_1+f_2}k_{\bz}||\\
&\leq ||H^*_{f_1(g_1-cg_2)}k_{\bz}||+||H^*_{g_1-cg_2}T_{f^*_1}k_{\bz}||+||T^*_{g_2}||\cdot||H^*_{cf_1+f_2}k_{\bz}||
\end{align*}
Then \eqref{52},\eqref{53} and Lemma \ref{*} imply $$\lim_{z\to m}||H^*_{cf_1+f_2}k_{\bz}||=0$$ and $$\lim_{z\to m}||H^*_{f_1(g_1-cg_2)}k_{\bz}||=0.$$
Lemma \ref{*lim} and \eqref{51} give $$\lim_{z\to m}||H^*_{g_1-cg_2}T_{f^*_1}k_{\bz}||=0.$$
Thus, \eqref{S2} holds.
This completes the proof.
\end{proof}

\bibliography{references}

\end{document}